\documentclass{article}

\usepackage{microtype}
\usepackage{graphicx}
\usepackage{subfigure}
\usepackage{booktabs} % for professional tables

\usepackage{amsmath}
\usepackage{amsfonts}
\usepackage{amssymb}
\usepackage{amsthm}

\newtheorem{theorem}{Theorem}[section]

\newtheorem{lemma}[theorem]{Lemma}

\newtheorem{remark}{Remark}

\usepackage{multicol}

\usepackage[accepted]{icml2019}

\usepackage[most]{tcolorbox}

\icmltitlerunning{Predict Globally, Correct Locally}

\begin{document}

\twocolumn[
\icmltitle{Predict Globally, Correct Locally:
		  Parallel-in-Time Optimal Control of Neural Networks}



\icmlsetsymbol{equal}{*}

\begin{icmlauthorlist}
\icmlauthor{Panos Parpas}{PP}
\icmlauthor{Corey Muir}{CM}

\end{icmlauthorlist}
\icmlaffiliation{PP}{Department of Computing, Imperial College London, London, United Kingdom}
\icmlaffiliation{CM}{Department of Computing, Imperial College London, London, United Kingdom}
\icmlcorrespondingauthor{Panos Parpas}{p.parpas@imperial.ac.uk}
\icmlkeywords{Machine Learning, ICML}
\vskip 0.3in
]



\printAffiliationsAndNotice{}  % leave blank if no need to mention equal contribution

\newcommand{\R}{\mathbb{R}}
\newcommand{\E}{\mathbb{E}}
\newcommand{\J}{\mathcal{J}}
\newcommand{\I}{\mathcal{I}}
\newcommand{\hP}{\widehat{P}}
\newcommand{\dx}{\partial}
\newcommand{\inp}[2]{\ensuremath{\langle{#1},{#2}\rangle}}

\begin{abstract}
 The links between optimal control of dynamical systems and neural networks have proved beneficial both from a theoretical and from a practical point of view. Several researchers have exploited these links to investigate the stability of different neural network architectures and develop memory efficient training algorithms. We also adopt the dynamical systems view of neural networks, but our aim is different from earlier works. We exploit the links between dynamical systems, optimal control, and neural networks to develop a novel distributed optimization algorithm. The proposed algorithm addresses the most significant obstacle for distributed algorithms for neural network optimization: the network weights cannot be updated until the forward propagation of the data, and backward propagation of the gradients are complete. Using the dynamical systems point of view, we interpret the layers of a (residual) neural network as the discretized dynamics of a dynamical system and exploit the relationship between the co-states (adjoints) of the optimal control problem and backpropagation. We then develop a parallel-in-time method that updates the parameters of the network without waiting for the forward or back propagation algorithms to complete in full. We establish the convergence of the proposed algorithm. Preliminary numerical results suggest that the algorithm is competitive and more efficient than the state-of-the-art.
\end{abstract}

\section{Introduction}\label{sec:Intro}
As transistors get smaller the amount of power per unit volume no longer remains constant as Dennard and co-authors predicted in 1974 \cite{dennard1974design}. After over thirty years, Dennard's scaling law ended in 2005, and  CPU manufacturers are no longer able to increase clock frequencies significantly \cite{koomey2011implications}.
The physical limitations of silicone-based microprocessors gave rise to computing architectures with many cores.
The reduction in cost and wide availability of multi-core processors completely revolutionized the field of neural networks, especially for computer vision tasks.
Therefore, the recent success of neural networks in many learning tasks can, in large part, be attributed to advances in computer architectures.  Alternative computing technologies (e.g., quantum computers) are not likely to be available soon. Thus any future advances in more efficient training of neural networks will come from algorithms that can exploit distributed computer architectures. 

Currently, the greatest obstacle in the development of distributed optimization algorithms for neural network training is the entirely serial nature of forward and backward propagation. The parameters of the neural network can only be updated after the forward propagation algorithm propagates the data from the first to the last layer and the backpropagation algorithm propagates the gradient information back to the first layer through all the layers of the network. This entirely serial nature of the training process severely hinders the efficiency of distributed optimization algorithms for deep neural networks. As a result, all the widely used frameworks for training neural networks only offer data parallelism, and the problem of layer-wise parallelism remains open. 

We propose a novel distributed optimization algorithm that breaks the serial nature of forward/backward propagation and allows for layer-wise parallelism. 
The proposed algorithm exploits the interpretation of neural networks, and Residual Neural Networks in particular, as dynamical systems. 
Several authors have recently adopted the dynamical systems point of view (see, e.g., \cite{0266-5611-34-1-014004,E2017,chen2018neural,li2017maximum} and Section \ref{sec:related work} for a discussion of related work). Reformulating the problem of Residual Neural Network (RNN) optimization as a continuous-time optimal comtrol problem allows us to model the layers of a neural network as the discretization time-points of a continuous-time dynamical system.
Thus RNN training can be interpreted as a classical optimal control problem. 
The results of this paper hold for Residual Neural Networks.
It is possible to extend the dynamical systems view to different architectures but in this paper we focus on Residual Neural Networks (RNNs).  
Using the interpretation of RNN training as an optimal control problem, we decompose the neural network across time (layers) and optimize the different sub-systems in parallel. 
%In Figure X we sketch the main idea of the paper. 
%The solution of the classification problem in Figure X is to transform the non-linearly separable co-centric ellipses in the top of the figure to a dataset that is linearly separable (the bottom of Figure X). 
%We view the training of a dynamical system as the problem of computing parameters ($U_t$) that achieve this transformation.
The central insight of this paper is that if the state and co-state (adjoint) of the dynamical system are approximately known in $N$ intermediate points, then we can parallelize the time dimension of the system and perform $N$ forward/backward propagations in parallel. 
This description justifies the name {\it``parallel-in-time"} method because we parallelize across the time dimension of the optimal control model. 
A significant challenge is to produce approximately correct state and co-state information for the optimal control problem. We address this problem by using a coarse discretization of the problem with a phase we call {\it global prediction} phase. 
The global prediction phase is followed by a {\it local correction} phase that attempts to improve the predicted optimal state and co-states of the control model solution.   

\subsection{Related Work \& Contributions}\label{sec:related work}
%Dual based methods, synthetic gradient, delayed gradient, maxi principle.
{\bf Layer-wise optimization of NNs.} Several authors have identified the limitations that backpropagation imposes on distributed algorithms. As a result, many approaches have been proposed including ADMM, block-coordinate descent \cite{zeng2018block} delayed gradients\cite{huo2018decoupled}, synthetic gradients\cite{jaderberg2017decoupled}, proximal \cite{lau2018proximal}, penalty based methods\cite{carreira2014distributed,huo2018decoupled}, and alternating minimization methods \cite{choromanska2018beyond}. 
Our method (especially the local correction phase) is related to the synthetic gradient approach in \cite{jaderberg2017decoupled}. The major differences between synthetic gradients of \cite{jaderberg2017decoupled} and our approach are that we exploit the dynamical-systems view and a multilevel discretization scheme to approximate the co-state (adjoint) variables quickly. 
We also establish the convergence of our method with weaker conditions than in \cite{jaderberg2017decoupled}. In Section \ref{sec:numerical experiments} we also show that our method outperforms even an improved version of the synthetic gradient method. 

{\bf The dynamical systems point of view.} 
Several authors have formulated neural network training as an optimal control problem in continuous time. 
For example, the authors in \cite{0266-5611-34-1-014004} have adopted this point of view to develop more stable architectures for neural networks. 
In \cite{li2017maximum} the authors used this approach to propose a maximum-principle based method as an alternative to backpropagation. The work in \cite{li2017maximum} also allows layer-wise optimization, but unfortunately, it was slower than stochastic gradient descent (SGD). 
In \cite{chen2018neural} the authors proposed to use classical ODE solvers to compute the forward and backward equations associated with NN training. We note that none of the existing approaches used the dynamical systems viewpoint to allow more parallelization in the optimization process. 
A recent exception is the method in \cite{gunther2018layer}. In \cite{gunther2018layer} the authors used the so-called parareal method to parallelize across the time (layer) dimension of neural networks. 
However, the parareal method is known to be problematic for non-linear systems \cite{gander2008nonlinear} and the efficiency reported by the authors is lower than our method and other methods such as synthetic and delayed gradients.

{\bf Contributions.} We exploit coarse and fine discretizations of continuous systems to compute predictions for optimal states/co-states of optimal control problems.
The state/co-state predictions allows us to develop a layer-wise (parallel-in-time) optimization algorithm for RNNs (Section \ref{sed: method description}).
We exploit the structure of the neural network training problem to show that the value function of the problem can be decomposed (Lemma \ref{seperable:lemma}). 
We also discuss the relationship between stochastic gradient descent, backpropagation and the co-state (adjoint) variables in optimal control (Lemma \ref{lemma: Back Propagation}).
We establish appropriate conditions under which the proposed algorithm will converge (Theorems \ref{theorem:Reduction in objective function} \& \ref{theorem: sgd}). We report encouraging numerical results in Section \ref{sec:numerical experiments}.

\section{Optimal Control \& Residual Neural Networks}\label{sec: Optimal Control + Residual Neural Networks}
The focus of this paper is Residual Neural Networks (RNNs) because they can achieve state-of-the-art performance in many learning tasks and can be reformulated as continuous-time optimal control problems. 
This reformulation offers several benefits such as stable classifiers, memory efficient algorithms and further insights into how and why they work so well (see, e.g., \cite{E2017,li2017maximum,0266-5611-34-1-014004,chen2018neural}). 
%From the dynamical systems point of view, the training of a RNN is equivalent to the problem of continuous-time optimal control. 
In this section, we review the dynamical systems viewpoint of RNN training. We review the links between the co-state (dual) variables of optimal control problems and the most widely used method to train neural networks, namely stochastic gradient descent and backpropagation (see Lemma \ref{lemma: Back Propagation}). In Lemma \ref{seperable:lemma} we make a simple observation that enables the decomposition of the optimal control problem across different initial conditions. This observation greatly simplifies the proof for a randomized version of our algorithm.  

For the origins, and advantages of the formulation below we refer to \cite{E2017,0266-5611-34-1-014004}, where training with residual neural networks is  reformulated as the following optimal control problem.
\begin{equation}
\begin{split}
V_0([X])=\min_{U} & \ \sum_{i=1}^m\Phi_i(X^i_T)+\int_{0}^T R(U_t)dt \\
& \dot{X}^i_t=f_t(X^i_t,U_t) \label{eq:original problem}\\
& X^i_0=x_i  \ i=1,\ldots,m.
\end{split}
\end{equation}
Where $\Phi_i:\R^n\rightarrow \R$ is a data fidelity term, and $R:\R^d\rightarrow \R$ is a regularizer. 
We note that the label for the $i^{\text{th}}$ data-point has been subsumed in the definition of $\Phi_i$.
The regularizer function could also be a function of the state ($X$) and time $t$. 
Regularization is typically used to prevent over-fitting and therefore applied to the parameters ($U$) of the network.
Hence we assume that the regularizer is only a function of the control parameter $U$. 
The notation $[X]=[x_1,\ldots,x_m]^\top\in\R^{md}$ is used to denote the initial conditions of the dynamics in \eqref{eq:original problem}.
We use $[X]_i\in\R^d$ to denote the $i^{\text{th}}$ element in $[X]$. 
The function $f:\R_+\times\R^d\times \R^u\rightarrow \R^d$ describes the activation function applied at time $t$. 
For example, for the  parameters $U=[W,b]$, the activation function may be defined as $\tanh(XW+b)$, where $\tanh$ is applied element wise. 
For the usual activation, regularization and data fidelity terms used in practice, the optimal control problem above is  well defined (see also \cite{MR1484411}). 

The function $V:\R^{md}\rightarrow \R$ defined in \eqref{eq:original problem} is called the {\it value function} of the problem. The value function, and its derivatives, will play a crucial role in this paper, and below we make a simple observation regarding the dimensionality of the value function. Using the principle of optimality \cite{MR3644954} we can rewrite the original model in \eqref{eq:original problem} as follows,
\begin{equation}
\begin{split}
V_0([X])=\min_{u} & \  \int_0^s R(U_t)dt+ \sum_{i=1}^m V_s([X_s]). \\
& \dot{X}^i_t=f_t(X^i_t,U_t) \\
& X^i_0=x_i \ i=1,\ldots,m. \label{eq:first half}
\end{split}
\end{equation}
Where $V_s([X])$ is defined as follows,
\begin{equation}
\begin{split}
V_s([X])=\min_{u} & \ \sum_{i=1}^m\Phi(X^i_T)+ \int_s^T R(U_t)dt  \\
& \dot{X}^i_t=f_t(X^i_t,U_t) \label{eq:second half} \\
& X^i_s=[X]_i \ i=1,\ldots,m.
\end{split}
\end{equation}
%Where above we used the notation $X_s=[X_s^1,\ldots,X_s^m]$.
%Note that the two sub-systems are intimately  connected. 
%The model in \eqref{eq:first half} requires knowledge of the optimal value function that can be obtained by solving \eqref{eq:second half}. While the second sub-system in \eqref{eq:second half} needs knowledge of the optimal initial conditions $[X]$ at time $s$. 
We take advantage of the structure of the problem to show that the value function can be written as a sum of identical functions from $\R^d$ to $\R$. 
\begin{lemma}\label{seperable:lemma} Let $V_s([X])$ denote the optimal value function of \eqref{eq:second half}. Then for any $0\leq s \leq T$, there exists 
a function $V_s:\R^{d}\rightarrow \R$ such that:
$V_s([X])=\sum_{i=1}^m V_s(x^i).
$
\end{lemma}
The result above simplifies the problem dramatically. 
It allows us to approximate the $md$-dimensional value function with a sum of $d$-dimensional functions. 

\subsection{Discretization}
To approximate the problem in \eqref{eq:original problem} we use a discretization scheme to obtain the following finite dimensional optimization problem. 
\begin{equation}
\begin{split}
\min_{U}\ & \J(U)=\sum_{i=1}^m \Phi_i(X^{i}_{T_\delta})+\sum_{j=0}^{T_\delta-1} R^\delta_j(U(t_j)) \\
& X^{i}(t_{j+1})=f^\delta_j(X^{i}(t_{j}),U(t_j)),  0\leq j\leq T_\delta -1
\\
& X^{i}_{0}=x_i, \ i=1,\ldots,m. \label{eq:finite discretization}
\end{split}
\end{equation} 
Where we used a discretization parameter $\delta$, such that with $\delta=T/T_\delta$ we obtain the discrete-time optimal control problem in \eqref{eq:finite discretization} with $T_{\delta}$ time-steps. 
We use the notation $f^\delta$ to denote an explicit discretization scheme of the dynamics in \eqref{eq:original problem}. 
For example, if an explicit Euler scheme is used then $f^\delta_j(X,U)=X+\delta f_j(X,U)$. 
When we use the simple explicit Euler scheme, with $\delta=1$ the formulation above reduces to the standard RNN architecture.
The deficiencies in terms of numerical stability of the explicit Euler scheme are well known in the numerical analysis literature \cite{butcher2016numerical}. 
In \citep{0266-5611-34-1-014004} the authors argue that the use of the explicit Euler scheme for RNNs explains numerical issues related to the training of NNs such as exploding/vanishing gradients and classifier instability. 
To resolve these numerical issues the authors in \cite{0266-5611-34-1-014004} propose to use stable discretization schemes. 
We follow the same line of reasoning, and show that a further advantage of appropriately defined discretization schemes is that they can be used to design convergent distributed algorithms to solve \eqref{eq:finite discretization}. 
Before we describe our approach we review serial-in-time methods and explain why it is difficult to design efficient distributed algorithms for \eqref{eq:finite discretization}.

{\bf Multilevel Discretization.} Since \eqref{eq:original problem} is a continuous time model we can approximate it using different levels of discretization.
To keep the notation simple and compact we assume that we want to solve \eqref{eq:original problem} using a step-size of $\delta$. 
This gives rise to the model in \eqref{eq:finite discretization}, with $T_\delta$ time-steps/layers. 
With our terminology the number of time-steps refers to the number of layers in the NN.
We use the terms time-steps and layers interchangeably.
We call the model with the a step-size of $\delta$ the  {\it fine model}. 
Later on we will take advantage of a coarse discretization of \eqref{eq:original problem}, and we refer to the resulting model as the {\it coarse model}. 
The time-step parameter of the coarse model is denoted by $\Delta$. 
We use $T_\Delta$ to denote the number of time-steps in the coarse model. 
The coarse model has less time-steps than the fine model, and it is therefore faster to optimize. 
As a practical example, suppose we are interested in the solution of \eqref{eq:original problem} using a fine model with 1024 steps (layers). 
Suppose we also use a coarse model with 64 steps/layers (for example). A forward/backward pass of the coarse model will roughly be 16 times faster. 
%This is just an example to illustrate the meaning of different levels of discretization for NNs.  
%We can, of course, have more than two discretization parameters but for ease of exposition we describe the method using only two levels, i.e. a fine and a coarse model. 
%In Section \ref{sec:numerical experiments} we describe how more than two discretization levels can be used in practice.
The use of multiple levels of discretization is a well known technique in the solution of optimal control problems, and has its origins in the multigrid literature\cite{briggs2000multigrid}. 
From the multigrid literature we will use the idea of interpolation operators in order to transform a trajectory from the coarse discretization grid to the fine grid (see \cite{briggs2000multigrid,0266-5611-34-1-014004}) 
%Any consistent discretization scheme (in the sense of \cite{MR1454128}) can be used to obtained the approximate model in \eqref{eq:finite discretization} (see Section \ref{sec:numerical} for additional discussion).
%In Section \ref{sec:numerical} we provide the details of the scheme we used in our experiments.
%\begin{remark} In the problem formulation in \eqref{eq:finite discretization} we should index the states and controls with $\delta$ i.e. $X^\delta$ instead of $X$, and $U^\delta$ instead of $U$. But it will be clear from context when we refer to the discrete and when we refer to the continuous time variables. We therefore use the simpler notation above.
%\end{remark}
\subsection{Serial-in-Time Optimization}
The most popular optimization algorithm for \eqref{eq:finite discretization} is batch stochastic gradient descent\cite{curtis2017optimization,bottou2018optimization}.  
Stochastic first order methods compute the gradient of the objective function in \eqref{eq:finite discretization}  with respect to the parameters $U$ using backpropagation.
In optimal control algorithms, the same gradient information is computed by solving backwards the {\it co-state} equations associated with the dynamics in  \eqref{eq:finite discretization}. 
The connections between co-states, adjoints, Lagrange multipliers and backpropagation are well known (see e.g. \cite{baydin2018automatic,lecun1988theoretical}).
We adopt the language used in scientific computing, and call the forward propagation the {\it forward solve}, and the backward propagation the {\it backward solve}.
The forward solve is specified in Algorithm \ref{alg:forward solve}. This algorithm plays the same role as forward propagation of conventional NN algorithms. 
The difference is that we do not rely on explicit Euler discretization but {\it we use a discretization scheme with a forward propagator that is stable and consistent with the dynamics in \eqref{eq:original problem}}. 
We use stable in the sense used in numerical analysis (see \cite{butcher2016numerical}) and consistent in the sense used in optimization (see \citep{MR1454128}). 
The backwards solve is specified in Algorithm \ref{alg:backward solve}. The purpose of the backward solve is to generate the information needed to compute the gradient of the objective function of \eqref{eq:finite discretization} with respect to the controls $U$. 
We use a stable and consistent scheme for the backward solve too.
After the forward and backward equations are solved, the information generated is used in some {\it algorithmic mapping} denoted by $\mathcal A$. This mapping generates a (hopefully) improved set of controls. 
A full iteration of a serial (in time) stochastic first order algorithm consists of a forward solve, a backward solve, followed by an update for $U$.
We state the standard serial-in-time algorithm in Algorithm \ref{alg:serial algorithm}.
In order to make our terminology more concrete we show that for a specific choice of $\mathcal A$, the procedure above reduces to the standard stochastic gradient method with backpropagation.
\begin{lemma}\label{lemma: Back Propagation}
Suppose that the algorithmic mapping in \eqref{eq:updateA} is defined as follows,
\[
\mathcal A(U,X,P)=U-\eta\left(\inp{\nabla_u f^\delta_t(X,U)}{P}+\nabla_u R(u)\right) 
\]
then Algorithm \ref{alg:serial algorithm} generates the same iterations as  batch stochastic gradient descent with a learning rate of $\eta$.  
\end{lemma} 
\section{A Parallel-in-Time Method}\label{sed: method description}
It is challenging to parallelize optimization algorithms for the model in \eqref{eq:finite discretization} because the backward solve cannot start before the forward solve finishes. 
Moreover, it is not possible to parallelize the forward solve because $X_{t+\delta}$ cannot be computed before $X_{t}$ is computed.
Similarly $P(t+\delta)$ must be computed before $P(t)$ in the backward solve.
Thus most algorithms only allow for data parallelism (e.g. across batches or in the calculation of $f$), but not across time/layers (see Section \ref{sec:related work} for some exceptions). 
%{\it But how can time be parallelized?} This is the question we address below. 
%\subsection{Overview of the Proposed Approach}
\begin{algorithm}[h]
\caption{{\tt Forward}$(\delta,X_{s},t_0,t_1,\{U(t)\}_{t=t_0}^{t=t_1})$}
\begin{algorithmic}\label{alg:forward solve}
\STATE $ t\leftarrow t_0, \  X(t)=X_{s}$
\STATE {\bf while}$(t \leq t_1$) {\bf do} \\
$
X({t+\delta})=f^\delta_t(X(t),U(t)), \ t\leftarrow t+\delta
$
\STATE {\bf return} $X(t), t_0\leq t \leq t_1$
\end{algorithmic}
\end{algorithm}
\vskip -0.5cm
\begin{algorithm}[h]
\caption{{\tt Backward}$(\delta,P_{e},t_0,t_1,\{U(t),X(t)\}_{t=t_0}^{t=t_1})$}
\begin{algorithmic}\label{alg:backward solve}
\STATE $ t\leftarrow t_1, \  P(t_1)=P_{e}$
\STATE {\bf while}$(t \geq t_0$) {\bf do} \\
$P({t-\delta})=-\langle \nabla_x f^\delta_t(X({t}),U({t}),P(t)\rangle$, $t\leftarrow t-\delta$  
%P^{i}_{k,T_\delta}=\nabla_x \Phi_i(X^i_{k,T_\delta}) \ i\in B_k.
\STATE {\bf return} $P(t), t_0\leq t \leq t_1$
\end{algorithmic}
\end{algorithm}
\vskip -0.4cm
\begin{algorithm}[h]
\caption{{\tt Serial-in-time$(\delta,0,T_\delta,\{U^0(t)\}_{t=t_0}^{t=T_\delta-1})$}}
\begin{algorithmic}\label{alg:serial algorithm}
\STATE Let $X^k(0)$ be a random sample from $[X]$.
\STATE $X^k(t)$
={\tt Forward$(\delta,X^k(0),0,T_\delta,U^k(t)),0\leq t\leq T_\delta$}
\STATE $P^k(T_\delta)=\nabla_x \Phi(X^k(T_\delta))$
\STATE $P^k(t)$={\tt Backward$(\delta,P^k(T),0,T,U^k(t)),0\leq t\leq T_\delta$}
\STATE {\tt Update control} for $0\leq t\leq T_\delta-1$
\begin{equation}
U^{k+1}(t)=\mathcal A(U^k(t),X^k(t),P^k(t+\delta)). 
\label{eq:updateA}
\end{equation}
\end{algorithmic}
\end{algorithm}

Suppose (somehow) we had an approximately optimal trajectory at some intermediate point $X^\star(s)$ and the corresponding co-state variable $P^\star(s)$. 
Suppose we also had two processors, Processor A and B, then we could potentially halve the cost of a full iteration of Algorithm \ref{alg:serial algorithm}.
We achieve this impressive reduction in time by using Processor A to do  a backward solve from $t=s$ to $t=0$ , followed by a forward solve from $t=0$ to $t=s$. 
In parallel, processor B is able to do a forward solve from time $t=s$ to time $T$, followed by a backward solve from time $t=T$ to time $t=s$. 
Thus, with two processors we can (potentially) halve the cost of a full iteration of any stochastic first-order method for \eqref{eq:finite discretization}. 
In reality, the reduction in time due to the extra processing power will not be halved (due to communication). 
% \begin{ ure}[H]
 % \vskip 0.2in
 % \begin{center}
 % \centerline{\includegraphics[width=\columnwidth]{figures/test1}}
 % \caption{}
 % \label{fig:cartoon}
 % \end{center}
 % \vskip -0.2in
 % \end{figure}
%\subsection{The two phases}
The difficult issue is how to compute the approximately optimal intermediate points $X^\star(s)$ and $P^\star(s)$? 
To address this question we proceed using two phases: a prediction phase, followed by a correction phase.
To be more precise, we first construct a coarse discretization of \eqref{eq:original problem} which we solve approximately using Algorithm \ref{alg:serial algorithm}. 
We call this phase the {\it global prediction} phase because it generates global information regarding the optimal trajectory. 
The global prediction phase generates useful information but it is not exact. 
To correct the prediction we have a second phase we call the {\it local correction} phase. 
The local correction is the time-parallel phase where we solve discretized versions of \eqref{eq:first half} to find better initial conditions for \eqref{eq:second half} in parallel. 
Below we explain these two phases in turn.

\subsection{Global Prediction Phase}
In the global prediction phase we construct an approximate solution of \eqref{eq:original problem} using a coarse discretization scheme. 
Our main assumption regarding the discretization scheme is that it provides a consistent approximation in the sense of \cite{MR1454128}. 
We use a large step-size $\Delta>\delta$, to approximate the model in \eqref{eq:original problem} with a finite dimensional optimization problem. 
We use a standard algorithm (e.g. Algorithm \ref{alg:serial algorithm}) to perform some iterations on the coarse model. 
%After $H$ iterations the algorithm will generate the following information $\{X_{k,j},P_{k,j}\}_{j=0}^{T_\Delta}$, $k=0,\ldots,H$. 
We call this phase the {\it global prediction} phase because we use the coarse model to quickly generate global information about the solution of the model in \eqref{eq:original problem}. 
In this context, the global information is contained in the forward trajectory (good initial conditions to initialize the local correction phase), and the backward trajectory (sensitivity information regarding the initial conditions at time $s$). 
We also obtain a good initial point for the parameters $U$. 
The coarse model is only used to generate predictions regarding the optimal state and co-states of the control problem. 
The predictions are corrected in parallel using the local correction phase described below. 

\subsection{Local Correction Phase}   
Using the information generated from the global prediction phase we split the original model into two sub-systems. 
The first subsystem is responsible for identifying the optimal initial conditions for the second subsystem. 
The second subsystem receives the initial conditions from the first subsystem and is responsible for solving the classification problem. 
The second subsystem also passes information back to the first subsystem in the form of sensitivity information (from the co-state variables).

The local correction phase is shown in Algorithm \ref{alg:gplc} for the case when the original model is decomposed into two sub-subsystems. 
The left column in Algorithm \ref{alg:gplc} describes the steps to optimize the first subsystem, and the right column describes the steps for the second subsystem. 
Note that the two sub-systems are solved in parallel. 
This is why we call the algorithm {\it parallel-in-time}. 
Our method computes the optimal parameters for time $t>s$ without waiting for information from the past $t\leq s$.

We first describe the work that Processor A performs (left column) on the first sub-system i.e. the optimization of \eqref{eq:finite discretization} from time $t=0$ to $t=s$ using a time-step of $\delta$.
%We use a consistent discretization scheme to solve the backward equation associated with the problem in \eqref{eq:finite discretization} but with a smaller ($\delta$ instead of $\Delta$) time-step than the coarse model. 
To start a backward solve at iteration $k$ from time $s$ we need $X^k(s)$, $P^k(s)$ and $U^k(s)$. 
In the first iteration of the local correction phase ($k=0$), we compute $X^0(s)$ and $U^0(s)$ by simple interpolation from the coarse model.
In subsequent iterations $X^k(s)$ and $U^k(s)$ are also available to the local correction phase of the first sub-system.
This is because the first subsystem has all the information required from time $0$ to time $s$.
Unfortunately, the co-state variable $P^k(s)$ is not available at time $s$ because to compute it we need to perform a backward solve from time $T$ to time $s$. 
%Note that standard interpolation schemes (see e.g. \cite{briggs2000multigrid}) work well for the state and control parameters. 
%However, it is not obvious how to adapt them to approximate the co-state variables.
Using coarse information only is not sufficient to build a good approximation for the co-state variables. 
Instead, we approximate $P^k(s)$ from state/co-state observations from the prediction phase, along with state/co-state observations collected from the second sub-system.
In the first iteration of the local correction phase we only have the observations from the prediction phase to approximate $P^0(s)$ (information from the second sub-system is not yet available).
We use $\mathcal I_k$ to denote the state/co-state pairs observed by the first sub-system at iteration $k$ of the local correction phase.
The prediction phase, after $H$ iterations of the coarse model,  produces the following information $\mathcal I_0=\{(X^i(s),P^i(s)), i=0,\ldots,H-1\}$. 
We use the state/co-state pairs observed so far to (approximately) solve the following regression problem,
\begin{equation}
\min_{A,B} L[\mathcal I_0]=\sum_{(X^i(s),P^i(s))\in\mathcal I_0} \|AX_i(s)+B-P_i(X_i(s))\|^2. \label{eq:regression co-state}
\end{equation}
Using the solution of the linear regression problem above we can approximate $P(s)=P(X(s))$ at any state $X(s)$ as $\hat P(s)\approx A^\star X(s)+B^\star$ (where $A^\star,
 B^\star$ are approximate solutions to the regression problem above). 
%Note that there is no need for anything more than least-squares regression due to the non-linear dynamics of $X(s)$.
After the backward solve finishes, we update the controls, do a forward-solve and pass the state $X_k(s)$ to the second subsystem. 
We next discuss how to solve the problem from time $s$ to time T, and how to update the information set $\mathcal I_0$.

In the first iteration of the local correction phase, Processor $B$ (right column) in Algorithm \ref{alg:gplc}, receives the approximate state $X^0(s)$, and controls  $\{U^{0}(t)\}_{t=s}^T$ from the global prediction phase.
Starting from $X^0(s)$ it performs a forward solve, a backward solve and updates its controls.
After the backward solve finishes, Processor B passes sensitivity information in the form of the co-state variables $P^0_{s}$ to Processor A.
Processor A then sets $\mathcal I_1=\mathcal I_0\cup (X_0(s),P_0(s))$.
The same steps are then repeated by both processors.
We use the notation $L[\mathcal I_k]$ when the regression problem in \eqref{eq:regression co-state} is solved with the information set $\mathcal I_k$. 
With a slight abuse of notation we write $\hP^k(s)=L[\mathcal I_k]$ to denote the approximate co-state information $P^k$ obtained using the solution of the regression problem in \eqref{eq:regression co-state}.

\begin{remark}[More than two sub-systems]
We described the algorithm using only two sub-systems. 
To use more than two sub-systems we observe that the second sub-system is a standard optimal control problem from time $s$ to time $T$. 
We can therefore use the same procedure we described in this section to divide the second subsystem into two. 
We can then continue to divide the system to as many sub-systems as required.
\end{remark}

\begin{remark}[Mini-batches and asynchronous computation]\label{rem:asynchronous}
To simplify the notation we did not describe the algorithm using mini-batches. 
In order to use mini-batches we just change Algorithms \ref{alg:forward solve} and \ref{alg:backward solve} to use mini-batches. 
We then change the notation so that $X^{k,i}(t)$ denotes the state at iteration $k$, batch $i$ at time $t$ (with similar notation for the control and co-state variables).
Finally, Algorithm \ref{alg:gplc} has a synchronization step after each sub-system completes a single iteration. This is how the algorithm was analyzed and implemented and we leave the asynchronous version for future work. 
\end{remark}

\vskip -0.4cm
\begin{algorithm}
\caption{Parallel-in-Time Optimization}
 \begin{algorithmic}\label{alg:gplc}
 \STATE {\bf Global Prediction:} Use Algorithm \ref{alg:serial algorithm} to make $m$ iterations of \eqref{eq:finite discretization} with step-size $\Delta$. 
 Initialize with $\{U^0_t,P^0_t,X^0_t]\}_{t=0}^T$ from global prediction phase.
 \end{algorithmic}
 \hspace{1cm} Processor A \hspace{2.5cm} Processor B
 \begin{tcolorbox}[sidebyside,size=tight,left=0pt,right=0pt,top=4pt,bottom=4pt,sidebyside gap=4pt] 
{\bf Backward solve:}\\
$\hP^k_s=\mathcal L[\mathcal I_k]$\\
{\tt Backward$(\delta,\hP^{k}_s,0,s,U^k)$}\\
{\bf Update:} \\
$U^{k+1}_t=\mathcal A(X^k_t,\hP^k_t)$ \\ % \ \ 0\leq t\leq s$
{\bf Forward solve:}\\
{\tt Forward$(\delta,X^k_0,0,s,U^k_t)$}\\
{\bf Synchronization:}
Send $X^{k}_s$ to Processor B.%\\
%Set $k\leftarrow k+1$.\\
%Go to {\bf Backward solve}.
\tcblower
\flushleft
{\bf Forward solve:}\\
{\tt Forward$(\delta,X^k_s,s,T_\delta,U^k_t)$}\\
{\bf Backward solve:}\\
$P^k_{T_\delta}=\nabla_x(X^k_{T_\delta})$\\
{\tt Backward$(P^{k}_{T_\delta},T_\delta,s,U^k)$}
{\bf Update:} \\
$U^{k+1}_t=\mathcal A(X^k_t,P^k_t)$ % \ \ 0\leq t\leq s$
{\bf Synchronization:}
Send $P^{k}_s$ to Processor A.%\\
%Set $k\leftarrow k+1$ and go to {\bf Forward solve}.
\end{tcolorbox}
\end{algorithm}

%\section{Optimality Conditions}
\section{Convergence Analysis}\label{sec:convergence analysis} 
In this section we summarize the theoretical  convergence results for Algorithm \ref{alg:gplc}. 
All proofs and technical lemmas appear in the supplementary material.
Algorithm \ref{alg:gplc} is quite general because we do not specify the algorithmic mapping in the update step, or the discretization scheme used to derive \eqref{eq:finite discretization}.
In order to keep the convergence analysis as close to the numerical implementation as possible we use the algorithmic mapping specified in Lemma \ref{lemma: Back Propagation}. 
It is possible to establish similar convergence results for other schemes too. For example, because our method can decompose a large network into smaller sub-networks, it might be possible to use second-order methods. We leave such refinements of the scheme to future work. Moreover, as discussed in Remark \ref{rem:asynchronous} the fact that the algorithm has a synchronization step simplifies the analysis.  This simplifying assumption allows us to analyze the algorithm as if it was run on a single processor.

The starting point of our analysis is the result in Lemma \ref{seperable:lemma}. It allows us to compute the optimal value function as the sum of $m$ $d-$dimensional functions as opposed to a single $md-$dimensional function. 
This is an important insight, because we know from the maximum principle \cite{MR1669395} that $P^\star_i(t)=-\nabla V_t(X^{\star}_i(t))$ at the optimal solution (where $i$ denotes the i$^{th}$ initial condition in \eqref{eq:original problem}).
Since $V$ is a function from $\R^d$ to $\R$, it follows that the adjoint is function from $\R^d$ to $\R^d$.
In fact, in the proof of Lemma \ref{lemma: Back Propagation} we show that,
 \[
 P^{k,\delta}_t=-\nabla_{x_t} \Phi(X^k_T).
 \]  
The equation above explains why the co-state variables provide sensitivity information for the first sub-system.
The main additional assumption we need to prove the convergence of Algorithm \ref{alg:gplc} (beyond the assumptions needed for any stochastic first-order methods) is that there exists an $\epsilon_p>0$ such that,
\begin{equation}
\| \hP^\delta(t)-P^\delta(t)\|\leq \epsilon_p\eta\|\hP^\delta(t)\|.\label{eq: bound on P assumption}
\end{equation}
We note that, at least in principle, such an $\epsilon_p$ is guaranteed to exist. 
The more interesting question is, of course, whether this constant is small or not.
The theoretical results below assume an arbitrary $\epsilon_p$. 
In practice, we found that this constant is small.
In Section \ref{sec:numerical experiments} we conjecture that this constant is small because we use data for several levels of discretization to construct $\hP$.
As was mentioned in the introduction, the regression step in the local correction phase has similarities with the the so called synthetic gradient method proposed in \cite{jaderberg2017decoupled}. The conceptual  differences between synthetic gradients and the proposed method were detailed in \ref{sec:related work}). However, the convergence analysis and assumptions are different too.
For example, in \cite{jaderberg2017decoupled} the authors assume that $\|S_g-\nabla J\|\leq\epsilon$ (where $S_g$ is the, so called, synthetic gradient).
Our assumptions are weaker.
In addition, the convergence analysis in \cite{jaderberg2017decoupled} is for gradient descent and not for stochastic gradient descent. 
Our first convergence result is for the case where the full gradient is used.
\begin{theorem}[Reduction in objective function]\label{theorem:Reduction in objective function}
Suppose that, $\eta_k\leq\frac{2}{2\epsilon_p+L}$. Then, $\mathcal J(U^{k+1})\leq\mathcal J(U^k)$ and in particular,
\[
\begin{split}
\mathcal J(U^{k+1})-\mathcal J(U^k)\leq 
-\sum_{j=0}^{T_\delta-1}\Bigg(
\theta_1 \|\dx_j f^\delta_k \|^2 \|\hP^k_{j+1}\|^2
\\
+ \theta_2 \|\dx_j f^\delta_k \| \|\hP^k_{j+1}\| \|\dx_j R^\delta_k\| + \theta_3\|\dx_j R^\delta_k\|^2
\Bigg),
\end{split}
\]
where the scalars $\theta_1\leq \theta_2 \leq \theta_3$ are positive and depend only on $L$ (the Lipschitz constant of $\J$) and $\epsilon_p$.
\end{theorem}
We note that if $\epsilon_p=0$, then Theorem \ref{theorem:Reduction in objective function} gives exactly the same result as the gradient descent method for \eqref{eq:finite discretization}.
Our next result deals with the case when the algorithm is run using mini-batches.
\begin{theorem}\label{theorem: sgd}
Suppose that the step-size in Algorithm \ref{alg:gplc} satisfies the following conditions,
\[
\sum_{k=1}^\infty \eta_k=\infty, \ 
\sum_{k=1}^\infty \eta^2_k<\infty.
\]
Then, $\lim_{M\rightarrow\infty}\frac{1}{H_M}\E\left( \sum_{k=1}^M \eta_k\| \nabla J(U^k)\|^2\right)=0$, 
where $H_M=\sum_{k=1}^M \eta_k$.
\end{theorem}
We note that for the theorem above to hold some additional restrictions on $\epsilon_p$ are required (see the on-line supplement for additional discussion).

\section{Numerical Experiments}\label{sec:numerical experiments}
In this section, we report preliminary numerical results for Algorithm \ref{alg:gplc}. 
This paper aims to establish a framework for time-parallel training of neural networks, and not to report on extensive numerical experiments. 
Still, it is essential to show that the algorithm is promising. 
We implemented the algorithm on a standard computer with a quad-core processor and 8GB of RAM. We decomposed the system into only two sub-systems.
We believe that the performance of our algorithm will improve with more than two sub-systems, and when ported to systems with a large number of cores.
Below we report results from three datasets. The first two datasets (co-centric ellipse, and swiss roll) are relatively low dimensional. To save space, we refer to \cite{0266-5611-34-1-014004} for a description of these two datasets.
The third dataset is the well known MNIST dataset.
 For large-scale problems, the cost of approximately solving a coarse model becomes a less significant part of the overall cost.  Therefore, we believe that the efficiency of our method will improve for large-scale problems.
 Indeed we see that for the MNIST dataset our algorithm achieves good speed-ups even for relatively shallow neural networks.
But such extensive numerical experiments are beyond the scope of the current paper.
The code is available from the authors' GitHub page (and in the on-line supplement during the review phase).

In Section \ref{sec:related work} we mentioned several approaches for layer (time) wise parallelization of neural networks.
Because the synthetic gradient method of \cite{jaderberg2017decoupled} is closely related to Algorithm \ref{alg:gplc} we only compare against a stable version of the synthetic gradient approach.
We also compare different variants of our method against the data-parallel implementation of SGD in Pytorch 0.4.1.
It will be interesting to compare all the different approaches described in Section \ref{sec:related work}, but this is not the aim of this paper.
In \cite{huo2018decoupled} the authors compared several layer-wise parallelization frameworks and concluded that their method, when tested on ResNet architectures on the CIFAR datasets, outperformed others and achieved speedups of 15\% to 50\%, without significant loss of accuracy.
Below we report similar results, but with higher parallel efficiency.
In \cite{huo2018decoupled} the authors compared their method against the synthetic gradient method in \cite{jaderberg2017decoupled}and found that their method significantly outperforms synthetic gradients.
Their implementation of synthetic gradients was based on network architectures that were shown to be unstable in \cite{0266-5611-34-1-014004}. So in our view, more careful numerical experiments are needed in order to decide the merits of the different approaches. However, these are beyond the scope of this paper.

{\bf Discretization schemes.} We considered two discretization schemes to derive the model in \eqref{eq:finite discretization}. The first one is based on an explicit Euler scheme and the second on symplectic integration using the Verlet method.
We chose the explicit Euler scheme because this scheme gives rise to the standard ResNet architecture.
We chose the Verlet scheme because it was shown to perform well in previous works \cite{0266-5611-34-1-014004}.
We also tested our method with and without the global prediction phase. When the global prediction phase is used we refer to our method as the multilevel parallel-in-time algorithm. 
When we do not use the global prediction phase then we call our method the single-level parallel-in-time method.
The regression step in our single-level, parallel-in-time method is similar to the decoupled neural interfaces method with synthetic gradients of \cite{jaderberg2017decoupled}. 
However, our single-level method is implemented with a stable discretization scheme. It is shown below that the discretization scheme makes a significant impact in the parallel efficiency and accuracy of the method.
The coarse model we used are exactly half the size of the fine model (e.g. if the fine model has 64 layers (steps) then the coarse model is constructed with 32 layers (steps)).

{\bf Accuracy of serial and parallel-time method.} The first observation from our results is that the 
the accuracy of our method (especially with the Verlet discretization scheme) is similar to the accuracy obtained with the data-parallel SGD method. It is clear from Figures \ref{fig:ellipse accuracy}, \ref{fig:swiss accuracy} and \ref{fig:mnist accuracy} (in the supplement) that the parallel-in-time method produces results with  similar accuracy as Stochastic Gradient Descent (SGD).

{\bf Speedup and parallel efficiency results.} Figure \ref{fig:swiss roll time} summarizes the speed-up obtained from our method against the data-parallel implementation of SGD in Pytorch. 
For our parallel-in-time method we include the time needed to solve the coarse (when used) and fine models in the speed-up calculations. 
We report results for the explicit discretization scheme without using the global prediction phase i.e. using a single-level discretization, and in Figure \ref{fig:swiss roll time} we refer to this method as the single-level Euler method.
We also report results using the multilevel scheme (i.e. using the global prediction phase) and the Verlet discretization scheme. 
In Figure \ref{fig:swiss roll time} we refer to this method as the multi-level Verlet method.
We observed similar speed-ups for both the ellipse and Swiss-Roll and so to save space we only report the results from the Swiss-Roll dataset in Figure \ref{fig:swiss roll time} (see Figure \ref{fig:ellipse time} in the on-line supplement for the ellipse dataset).  
From Figure \ref{fig:swiss roll time} we see that for  relatively small data-sets (e.g. the ellipse and Swiss-Roll data-sets) there is a cut-off point (around 32 to 64 layers) after which our method is faster than the data parallel implementation. 
Since we only use two processors, we observe that, for deep networks, our method achieves an efficiency of about 75\%. The efficiency of our algorithm is substantially better than the speed-up efficiency of 50\% reported in \cite{huo2018decoupled}.
Our results compare even more favorably with speed-up efficiencies of 3-4\% reported in \cite{gunther2018layer} for an alternative parallel-in-time method.
 \begin{figure}[t]
 \vskip -0.1in
 \begin{center}
 \centerline{\includegraphics[width=\columnwidth]{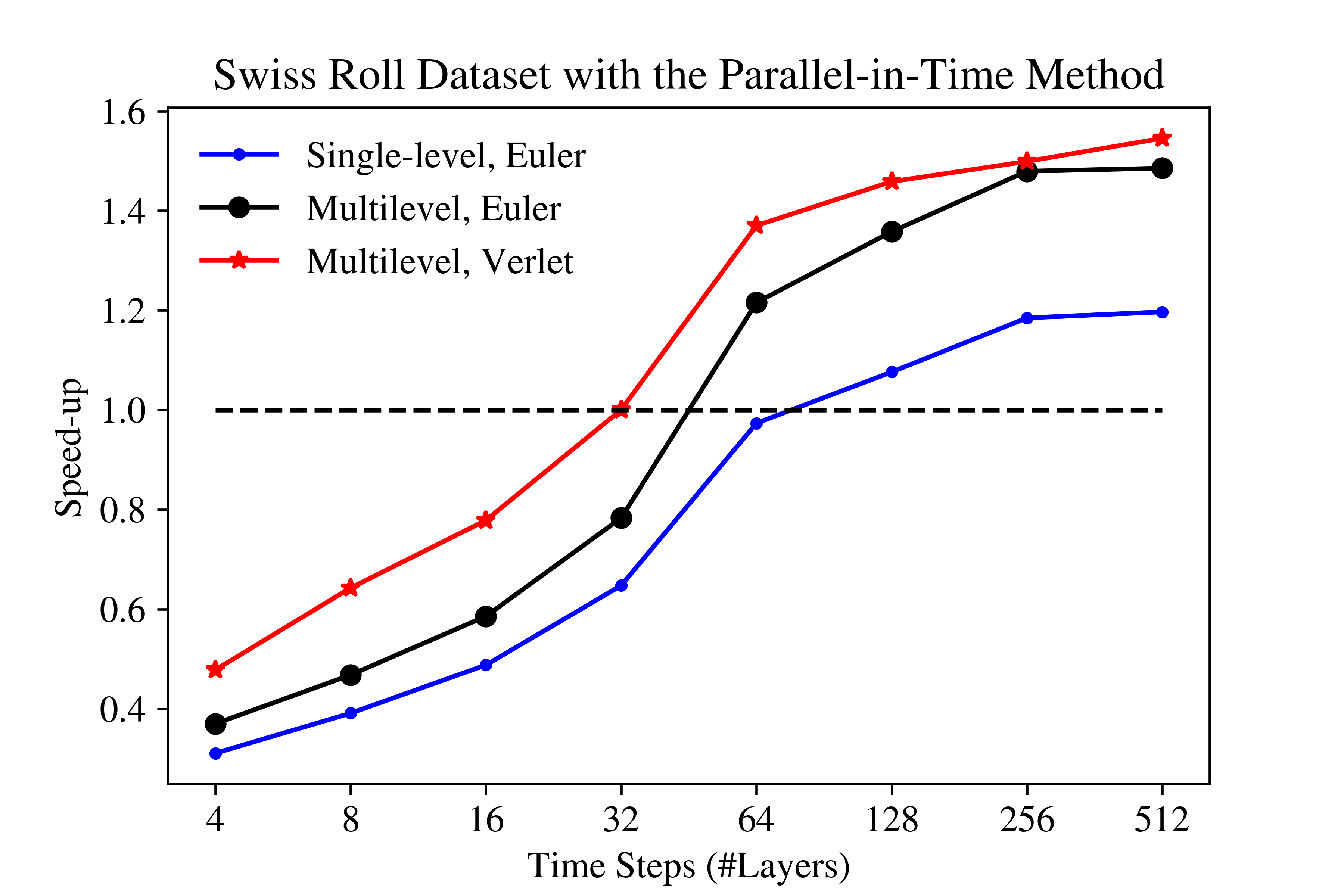}}
 \caption{Speed-up for the Swiss-Roll dataset over SGD. 
 %The results suggest an efficiency of 75\% (for large networks) which is much greater than any existing layer-parallel methods for training NNs.
 }
 \label{fig:swiss roll time}
 \end{center}
 \vskip -0.4in
 \end{figure}

In Figure \ref{fig:mnist times} we summarize the results for the MNIST dataset. For the MNIST dataset we see that our method is faster than the SGD even for relatively shallow networks (4 layers).
From these results we see that our method achieves much better speed-ups on MNIST than other speed-ups reported in the literature.
Moreover, the efficiency of over 75\% for the MNIST dataset suggest that the communication overheads of our method are small. 
These results validate our claim that our method will have an advantage over existing methods for larger models. 
The reason is that, for large models, the time spent solving the coarse model is a small proportion of the total solution time.  

\begin{figure}[h]
  \vskip -0.1in
 \begin{center}
 \centerline{\includegraphics[width=\columnwidth]{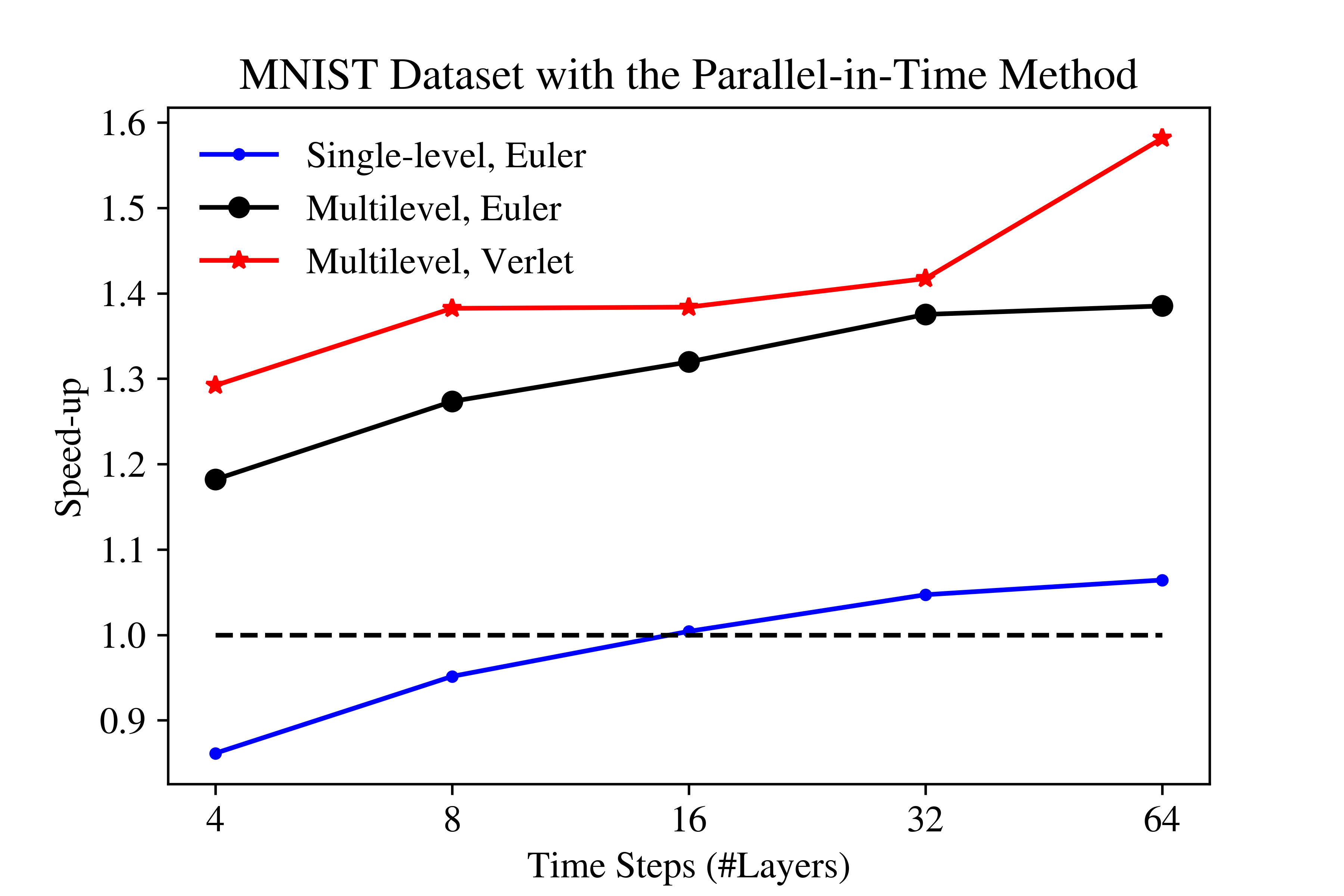}}
 \caption{Speed-up for the MNIST dataset over SGD. 
 %data-parallel implementation of SGD. 
 %The results suggest an efficiency of over 75\% (for large networks) which is much greater than any existing layer-parallel methods for training NNs.
 }
 \label{fig:mnist times}
 \end{center}
 \vskip -0.3in
 \end{figure}

{\bf Increased stability due to the global prediction phase.}
To test the impact of the global prediction phase using the coarse model we report the mean-square errors from the regression step in the backward solve of Algorithm \ref{alg:gplc}.
When the global prediction step is not used our method is similar to the synthetic gradient method from \cite{jaderberg2017decoupled}.
We see from Figure \ref{fig:errors sg} that when the global prediction phase is not used the mean-square error of the regression step varies significantly in the first 20 iterations before converging to a non-zero value.
When the global prediction phase is used, we can see from Figure \ref{fig:errors sg} (left y-axis) that the MSE is an order of magnitude lower and eventually converges to zero.
These results explain why our method is so efficient. These results also provide empirical validation for the assumption in \eqref{eq: bound on P assumption} required to prove the convergence of our method. The results in Figure \ref{fig:errors sg} are for the Swiss-Roll dataset with 512 layers using the Verlet discretization scheme. Similar behavior was observed in the other models.
\begin{figure}[h]
 \vskip -0.3in
 \begin{center}
 \centerline{\includegraphics[width=\columnwidth]{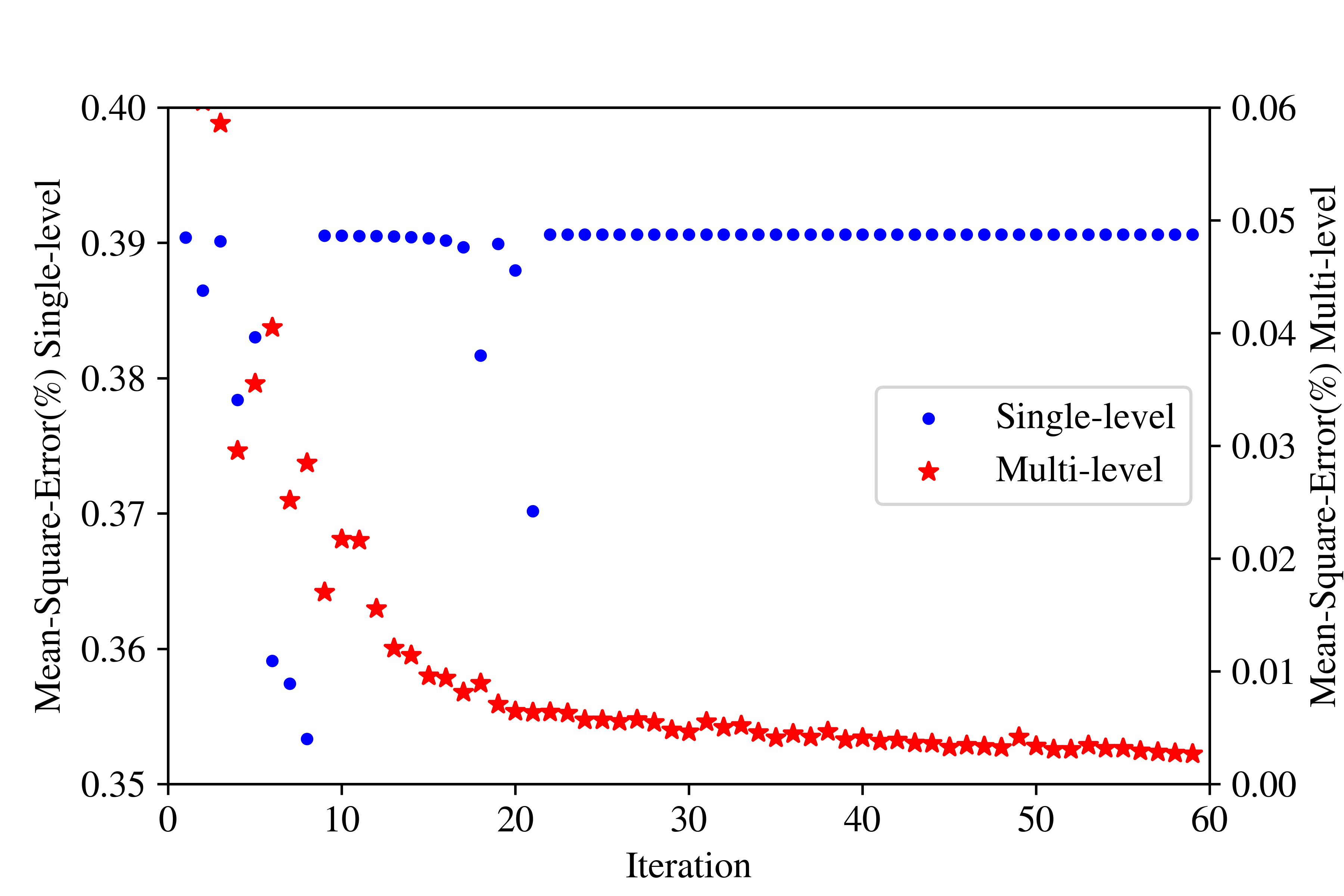}}
 \caption{Mean Square Errors of regression step in Algorithm \ref{alg:gplc}% is an order of magnitude lower and converges to zero when global prediction phase is performed.
 }
 \label{fig:errors sg}
 \end{center}
 \vskip -0.5in
 \end{figure}
\section{Conclusions}\label{sec:conclude}
We proposed a novel parallel-in-time distributed optimization method for neural networks. The method exploits the dynamical systems view of residual neural networks to parallelize the optimization process across time (layers). 
We discussed how to take advantage of multilevel discretization schemes in order to predict the optimal states and co-states of the control model. 
We established the convergence of our method. Our initial numerical results suggest that the proposed method has the same accuracy as Stochastic Gradient Descent, reduces computation time significantly and achieves higher parallel efficiency than existing methods. The method can be improved in several ways including an asynchronous implementation, using more than two discretization levels and decomposing the original network to several components. More detailed numerical experiments are needed to understand the behavior of the method for large datasets, but the initial efficiency results are extremely encouraging.
\bibliography{predict_correct.bib}
\bibliographystyle{icml2019}

\clearpage
\onecolumn 
\section{
Supplementary Material for: Predict Globally, Correct Locally:
      Parallel-in-Time Optimal Control of Neural Networks 
      }
\subsection{Notation}
In this section we briefly define our notation.
We use $X^{\xi,k}_\delta(t)$ represent the state at time $t$, iteration $k$, batch $\xi$, with discretization parameter $\delta$. 
When it is clear from context we drop $\xi$ and $\delta$ and write $X^k_t$ instead. When it is relevant to identify or sum over the different time-steps we use the notation 
\[
X^k_j=X^{k}(t_j).
\]
We use $X^k$ to denote the vector form of $X^k_j$. 
We use the same conventions for the exact co-state variables $(P^k_t)$, the approximate co-states $(\hP^k_t)$ and the control parameters/weights $(U^k_t)$. Below all norms are $\ell_2$ norms. 

Algorithm \ref{alg:gplc} has a synchronization step.
This assumption implies that if the algorithmic mapping in Lemma \ref{lemma: Back Propagation} is used then the control parameters $U$ are updated as follows,
\begin{equation}
U^{k+1}=U^{k}-\eta_k G(U^k) \label{eq:simple iteration}
\end{equation}
where,
\[
G(U^k)=-\inp{\nabla_{x}f(X^k,U^k)}{\hP^k}+ \nabla_{u}R(U^k).
\]
We note that if gradient descent is used to solve \eqref{eq:finite discretization} then Algorithm \ref{alg:serial algorithm} reduces to,
\[
U^{k+1}=U^{k}-\nabla_u \J(U^k)
\]
where $\J(U)$ is the objective function of \eqref{eq:finite discretization}, and the gradient $\nabla_u \J$ is calculated using backpropagation. Below we also use the following short-hand notation,
\[
\dx_j f^\delta_k =\nabla_{u_j}f^\delta(X^k_j,U^k_j)
\]
\[
\dx_j R^\delta_k =\nabla_{u_j}R^\delta(U^k_j)
\]
\[
\dx_j \mathcal J_k=\nabla_{u_j}\mathcal J(U^k)
\]
We note that with the simplified notation above we have $G_j(U^k)=-\inp{\dx_j f^\delta_k}{\hP^k_{j+1}}+\dx_jR^\delta_k$.

When the algorithm is run with mini-batches, we use 
$\E[A]$ to denote the total expectation of the random variable $A$, and $\E[A~|~\I_k]$ is conditional on the information available up to and including iteration $k$. 

\subsection{Assumptions}
In this section we outline our main assumptions. 
\begin{enumerate}
\item The objective function ($\J$) in \eqref{eq:finite discretization} has Lipschitz continuous gradient. We denote the Lipschitz constant with $L$.
\item The problems in \eqref{eq:original problem} and \eqref{eq:finite discretization} have a finite solution.
\item The discretization scheme used to obtain the \eqref{eq:finite discretization} from \eqref{eq:original problem} is stable (in the sense of \cite{butcher2016numerical}) and consistent in the sense of \cite{MR1454128}. 
\item The error in the adjoint calculation satisfies the inequality,
\begin{equation}
\|P^k_j-\hP^k_j\|\leq \epsilon_p\eta_k \| \hP^k_j\| \label{eq:assume co-state}
\end{equation}

\end{enumerate}

\subsection{Proofs for Section \ref{sec: Optimal Control + Residual Neural Networks}}
\begin{proof}(Of Lemma \ref{seperable:lemma}).\\
This result can easily be established by induction.
At the terminal time $s=T$, we must have,
\[
V_s([X])=\sum_{i=1}^m \Phi(x^i).
\]
Therefore, by taking $V_T(x)\triangleq \frac1m \Phi(x)$ we establish that the Lemma is true for $s=T$. Suppose that the Lemma is true for some $0<t+\delta t<T$, we show that it is true for $t$ also. Let $u^\star(t)$ denote an optimal solution, then by assumption we have,
\[
\begin{split}
V_{t}([X])&=\int_{t}^{t+\delta t} R(u^\star_s)ds+V_{t+\delta t}([X]) \\
&=\int_{t}^{t+\delta t} R(u^\star_s)ds + \sum_{i=1}^m V_{t+\delta}(x^i).
\end{split}
\]
Let $V_{t}(x)\triangleq\frac1m(\int_{t}^{t+\delta t} R(u^\star_s)ds + \sum_{i=1}^m V_{t+\delta}(x))$, and the result follows.  
\end{proof}

\begin{proof}(Of Lemma \ref{lemma: Back Propagation}).\\ 
At time $t=T=T_\delta$, it follows from the boundary condition enforced by Algorithm \ref{alg:backward solve} that $P^k_T=-\nabla_{X_T}\Phi(X^k_T)$. If we take one step of the backward solve then for $t=T-\delta$ we obtain,
\[
\begin{split}
P^k_{t}&=-\inp{\nabla_{x_{t}}f^\delta(X^k_{t},U^k_{t})}{P^k_T}\\
&=\inp{\nabla_{x_{t}}f^\delta(X^k_{t},U^k_{t})}{\nabla_{X_T}\Phi(X^k_T)}\\
&=-\nabla_{x_{t}}\Phi(f^\delta(X^k_{t},U^k_{t}))\\
&=-\nabla_{x_{t}}\Phi(X^k_{T}).
\end{split}
\]
Using the same argument recursively we conclude that,
\begin{equation}
P^k_{t}=-\nabla_{x_{t}}\Phi(X^k_{T}), \ 0 \leq t \leq T. \label{eq: P relation to phi}
\end{equation}
Using the preceding equation we obtain,
\[
\begin{split}
\dx_j J_k&=\nabla_{u_j}\Phi(X^k_T)+\nabla_{u_j}R^\delta(u^k_j) \\
&=\inp{\nabla_{u_j}(X^k_{j+1})}{\nabla_{x_{j+1}} \Phi(X^k_T)}+\nabla_{u_j}R^\delta(u^k_j)\\
&=-\inp{\nabla_{u_j}f^\delta(X^k_j,U^k_j)}{P^k_{j+1}}+\nabla_{u_j}R^\delta(u^k_j).
\end{split}
\]
It follows that,
\[
U^{k+1}_j=U^k_j-\eta \dx_j J_k=\mathcal A(U^k_j,X^k_j,P^k_{j+1}), 
\]
and we conclude that SGD with a learning rate of $\eta$ produces the same iterations as Algorithm \ref{alg:serial algorithm}.
\end{proof}

%\subsection{Case I: Deterministic Error Bound, Gradient descent with full batch }

% \begin{lemma}[Lower bound on $\hP_k$]
% \begin{equation}
% \| \hP_{j+1}^k\| \geq \frac{\|\dx_j \mathcal J_k\|-\|\dx R^\delta_k\| }
% {(1+\epsilon_p)\|\dx_j f^\delta_k\|}
% \label{eq: bound P}
% \end{equation}
% \end{lemma}
% \begin{proof}
% \end{proof}

\subsection{Proofs for Section \ref{sec:convergence analysis}}

\begin{proof}(Of Theorem \ref{theorem:Reduction in objective function})
Because $\mathcal J$ has a Lipschitz continuous gradient it follows from \eqref{eq:simple iteration},
\[
\begin{split}
\mathcal J(U_{k+1})\leq& J(U_{k})+\sum_{j=0}^{T_\delta-1}\left(
-\eta_k\inp{\dx_j\mathcal J_k}{G_j(U^k)}+\frac{\eta^2_k L}{2}\|G_j(U^k)\|^2
\right)\\
=& \mathcal J(U_{k})+\sum_{j=0}^{T_\delta-1}\left(
-\eta_k\inp{\dx_j\mathcal J_k-G_j(U^k)}{G_j(U^k)}+\frac{\eta^2_k L-2\eta_k}{2}\|G_j(U^k)\|^2
\right)\\
=& \mathcal J(U_{k})+\sum_{j=0}^{T_\delta-1}\left(
\eta_k\inp{ \inp{\dx_j f^\delta_k}{P^k_{j+1}-\hP^k_{j+1}}  }{G_j(U^k)}
+\frac{\eta^2_k L-2\eta_k}{2}\|G_j(U^k)\|^2
\right)\\
\leq& \mathcal J(U_{k})+\sum_{j=0}^{T_\delta-1}\Bigg(
\Big(\epsilon_p\eta_k^2+\frac{\eta^2_k L-2\eta_k}{2}\Big)\|\dx_j f^\delta_k \|^2 \|\hP^k_{j+1}\|^2
+ \Big(\epsilon_p\eta_k^2+\eta^2_k L-2\eta_k\Big) \|\dx_j f^\delta_k \| \|\hP^k_{j+1}\| \|\dx_j R^\delta_k\| \\
& \hspace*{10cm} + \frac{\eta^2_k L-2\eta_k}{2}\|\dx_j R^\delta_k\|^2
\Bigg)
\end{split}
\]
Since $\eta_k\leq\frac{2}{2\epsilon_p+L}$ the result follows.

\end{proof}
%\subsection{Case II: Probabilistic Error Bound, Gradient descent with batch}
Next we discuss the case when Algorithm \ref{alg:gplc} is run using mini-batches. In this case, the iteration in \eqref{eq:simple iteration} is replaced with the following,
\[
U^{k+1}=U^{k}-\alpha_k G(U^k,\xi)
\]
where $G(U^k,\xi)$ denotes a randomized version of $G(U^k)$. The following assumption is standard regrading the sampling process (see e.g. \cite{bottou2018optimization}),
\begin{equation}
\inp{\nabla \mathcal J_k(U^k)}{\E[\nabla\J_k(U^k,\xi)~|~\I_k]}\geq \mu \|\nabla \mathcal J_k(U^k)\|^2. \label{eq:assume first moment I }
\end{equation}
We now establish some technical lemmas that are needed for the proof of Theorem \ref{theorem: sgd}.
\begin{lemma}\label{lemma:bound first moment}
Suppose that 
\[
\E[\|\nabla_uf^{\delta}(X^k,U^k,\xi)\| \| \hat P(\xi)\|~|~\I_k] 
\leq M_1 \|\nabla \mathcal J_k(U^k)\|
\]
Then,
\[
 \inp{\nabla \mathcal J_k(U^k)}{\E[G(U^k,\xi)~|~\I_k]} \geq
 (\mu -\epsilon_p M_1) \|\nabla \mathcal J_k(U^k)\|^2   
\]
\end{lemma}
\begin{proof}
It follows from \eqref{eq:assume first moment I }
\[
\begin{split}
\mu \|\nabla \mathcal J_k(U^k)\|^2 &\leq  
\inp{\nabla \mathcal J_k(U^k)}{\E[-\inp{\nabla_uf^{\delta}(X^k,U^k,\xi)}{P^k(\xi)}+\nabla_u R^\delta(U^k)~|~\I_k]}\\
&= \inp{\nabla \mathcal J_k(U^k)}{\E[-\inp{\nabla_uf^{\delta}(X^k,U^k,\xi)}{P^k(\xi)-\hP^k(\xi)}-\inp{\nabla_uf^{\delta}(X^k,U^k,\xi)}{\hP^k(\xi)} +\nabla_u R^\delta(U^k)~|~\I_k]}\\
&\leq \inp{\nabla \mathcal J_k(U^k)}{\E[G(U^k,\xi)~|~\I_k]}
+\epsilon_p  \|\nabla \mathcal J_k(U^k)\|\E
[
\|\nabla_uf^{\delta}(X^k,U^k,\xi)\| \| \hat P(\xi)\|
~|~\I_k]  \\
&\leq \inp{\nabla \mathcal J_k(U^k)}{\E[G(U^k,\xi)~|~\I_k]}
+\epsilon_p M_1 \|\nabla \mathcal J_k(U^k)\|^2,
\end{split}
\]
and by re-arranging the inequality above, the result follows.
\end{proof}
\begin{lemma}\label{lemma:bound second moments}
Suppose that
\begin{equation}
\E[\|\inp{\nabla_uf^{\delta}(X^k,U^k,\xi)}{\hP^k(\xi)}\|^2+\|\nabla_u R^\delta(U,\xi)\|^2~|~\I_k]\leq M_2+M_3\|\nabla_u\J(U^k)\|^2 \label{eq: bound second moments}
\end{equation}
Then, 
\[
\E[\J(U^{k+1})~|~\I_k]-\J(U^k)\leq 
-\eta_k(\mu-\epsilon_pM_1-\eta_kM_3L)  \|\nabla \mathcal J_k(U^k)\|^2   
+L\eta_k^2 M_2.
\]
\end{lemma}

\begin{proof}
\begin{equation}
\E[\J(U^{k+1})~|~\I_k]-\J(U^k)\leq -\eta_k\inp{\nabla \J(U^k)}{\E[G(U^k,\xi)~|~\I_k]}+\frac{L\eta_k^2}{2}\E[\|G(U^k,\xi)\|^2~|~\I_k] \label{eq: descent lemma J-1}
\end{equation}
We can bound the first term using Lemma \ref{lemma:bound first moment}
% We  bound the second term as follows,
% \[
% \begin{split}
% \|G(U^k,\xi)\|&=\|-\inp{\nabla_uf^{\delta}(X^k,U^k,\xi)}{\hP^k(\xi)}
% +\nabla_u R^\delta(U^k,\xi)\| \\
% &\leq 
% \|\inp{\nabla_uf^{\delta}(X^k,U^k,\xi)}{\hP^k(\xi)}\|+\|\nabla_u R^\delta(U^k,\xi)\|
% \end{split}
% \]
% Taking squares on both sides, using the inequality $(a+b)^2\leq 2(a^2+b^2)$ and taking expectations we find,
We  bound the second term as follows,
\[
\begin{split}
\|G(U^k,\xi)\|&=\|-\inp{\nabla_uf^{\delta}(X^k,U^k,\xi)}{\hP^k(\xi)}
+\nabla_u R^\delta(U^k,\xi)\| \\
&\leq 
\|\inp{\nabla_uf^{\delta}(X^k,U^k,\xi)}{\hP^k(\xi)}\|+\|\nabla_u R^\delta(U^k,\xi)\|
\end{split}
\]
Taking squares on both sides, using the inequality $(a+b)^2\leq 2(a^2+b^2)$ and taking conditional expectations on both sides we find,
\[
\begin{split}
\E[\|G(U^k,\xi)\|^2~|~\I_k]
&\leq 
2\left(\E[\|\inp{\nabla_uf^{\delta}(X^k,U^k,\xi)}{\hP^k(\xi)}\|^2~|~\I_k]
+\E[\|\nabla_u R^\delta(U^k,\xi)\|^2~|~\I_k]\right)\\
&\leq  2(M_2+M_3\|\nabla_u\J(U^k)\|^2).
\end{split}
\]
%Taking the conditional expectation on both sides and using \eqref{eq: bound second moments} we find,
%\[
% \begin{split}
% \E[\|G(U^k,\xi)\|~|~\I_k]
% \leq 
% M_2+M_3\|\nabla\J(U^k)\|,
% \end{split}
% \]
Using the bound above in \eqref{eq: descent lemma J-1}, and Lemma \ref{lemma:bound first moment} we obtain,
\[
\E[\J(U^{k+1})|\I_k]-\J(U^k)\leq 
-\eta_k(\mu-\epsilon_pM_1-\eta_kM_3L)  \|\nabla \mathcal J_k(U^k)\|^2   
+L\eta_k^2 M_2,
\]
as claimed.
\end{proof}

\begin{lemma}\label{lemma:fixed step size}
Suppose that Algorithm \ref{alg:gplc} is run with a fixed step-size $\bar\eta$ such that,
\begin{equation}
0<\bar\eta\leq \frac{\mu-\epsilon_p M_1}{2M_3L},
\end{equation}
then after $N$ iterations the following holds,
\[
\frac1N \sum_{i=1}^{N} \|\nabla\J(U^k) \|^2
\leq \frac{2L\bar\eta M_2}{\mu}+\frac{2(\J(U^N)-\J(U^1))}{\bar\eta N \mu}
\]
\end{lemma}
\begin{proof}
Taking expectations on the bound obtained in Lemma \ref{lemma:bound second moments} and using the law of total expectation we obtain,
\[
\begin{split}
\E[\J(U^{k+1})]-\E[\J(U^k)]&\leq 
-\bar\eta(\mu-\epsilon_pM_1-\bar\eta_kM_3L)  \E[\|\nabla \mathcal J_k(U^k)\|^2]   
+L\bar\eta_k^2 M_2 \\
& \leq -\frac{\bar\eta\mu}{2} \E[\|\nabla \mathcal J_k(U^k)\|^2]
+L\bar\eta^2 M_2.
\end{split}
\]
Summing from $k=1$ to iteration $N$ we obtain,
\[
\J^\star-\J(U^1)\leq J(U^M)-J(U^1)
\leq -\frac{\bar\eta\mu}{2}\sum_{k=1}^{N}\E[\|\nabla \mathcal J_k(U^k)\|^2]+NL\bar\eta^2 M_2.
\]
Rearranging the inequality above we obtain the required result. 
\end{proof}

\begin{proof}(Of Theorem \ref{theorem: sgd})
This result can be established by observing that the conditions imposed on the step-size imply that $\eta_k\rightarrow 0$. Therefore, for $k$ sufficiently large the assumption on the step-size in Lemma \ref{lemma:fixed step size} holds. The rest of the proof is the same as the proof of Lemma \ref{lemma:fixed step size}.
\end{proof}

\subsection{Additional Figures for Section \ref{sec:numerical experiments}}

\begin{figure}[H]
 \begin{center}
 \centerline{\includegraphics[scale=0.75]{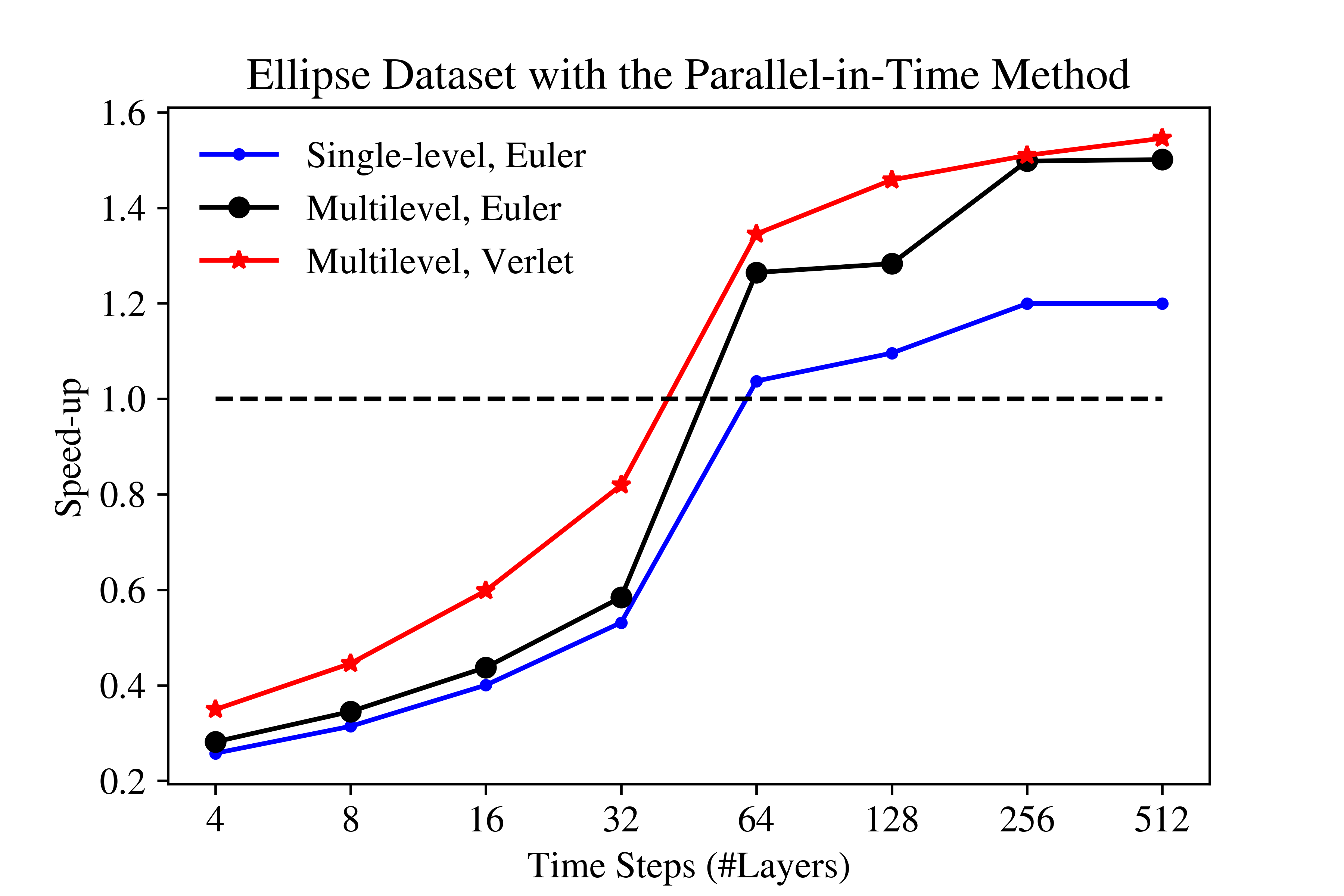}}
 \caption{Speed-up for the Ellipse dataset over a data-parallel implementation of SGD. Like the Swiss-Roll dataset the results suggest an efficiency of 75\% (for large networks) which is much greater than existing layer-parallel methods for training NNs.}
 \label{fig:ellipse time}
 \end{center}
 \end{figure}

\begin{figure}[H]
 \begin{center}
 \centerline{\includegraphics[scale=0.75]{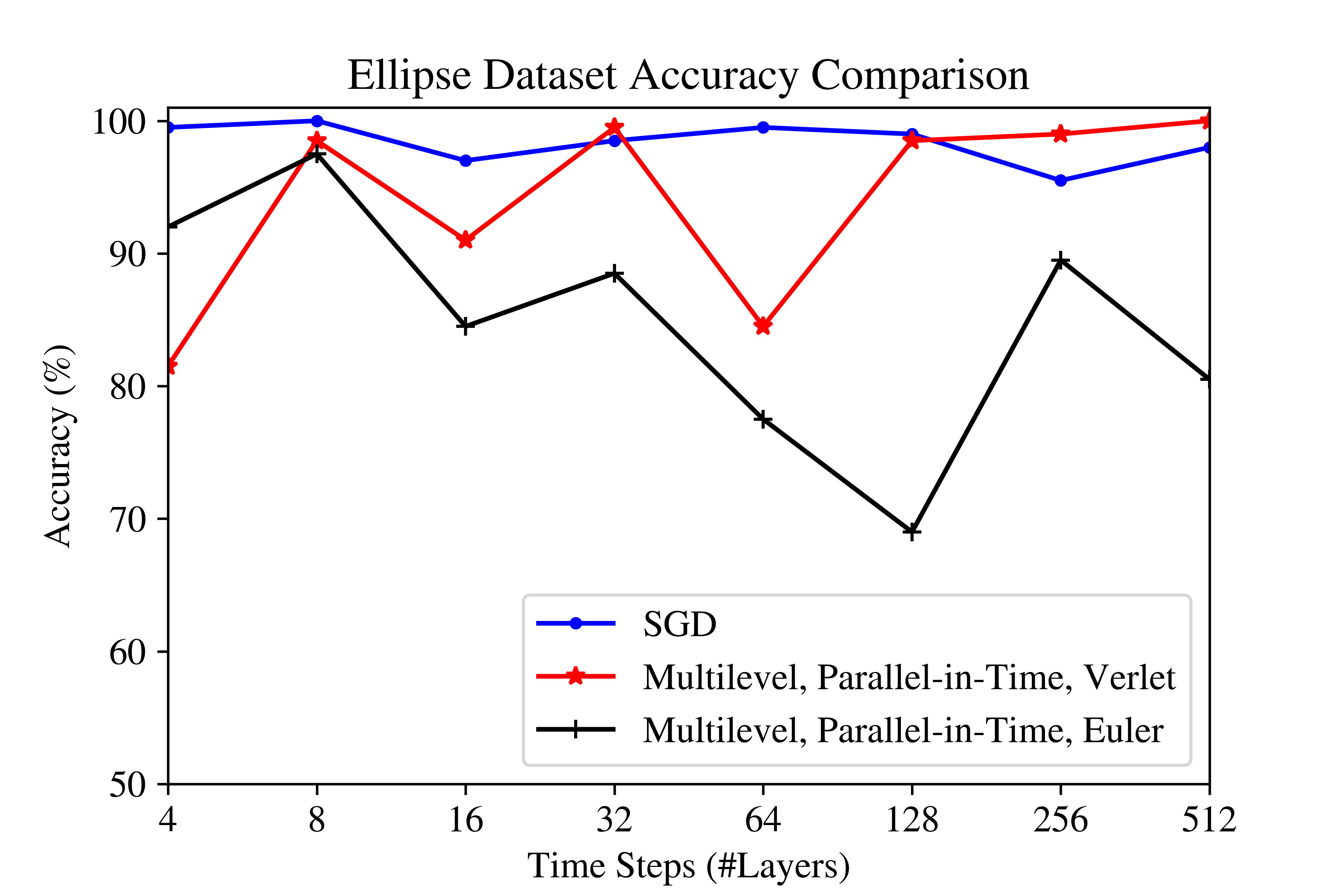}}
 \caption{Accuracy for the Ellipse dataset when compared with a  data-parallel implementation of Stochastic Gradient Descent (SGD).}
 \label{fig:ellipse accuracy}
 \end{center}
 \end{figure}

\begin{figure}[H]
 \begin{center}
 \centerline{\includegraphics[scale=0.75]{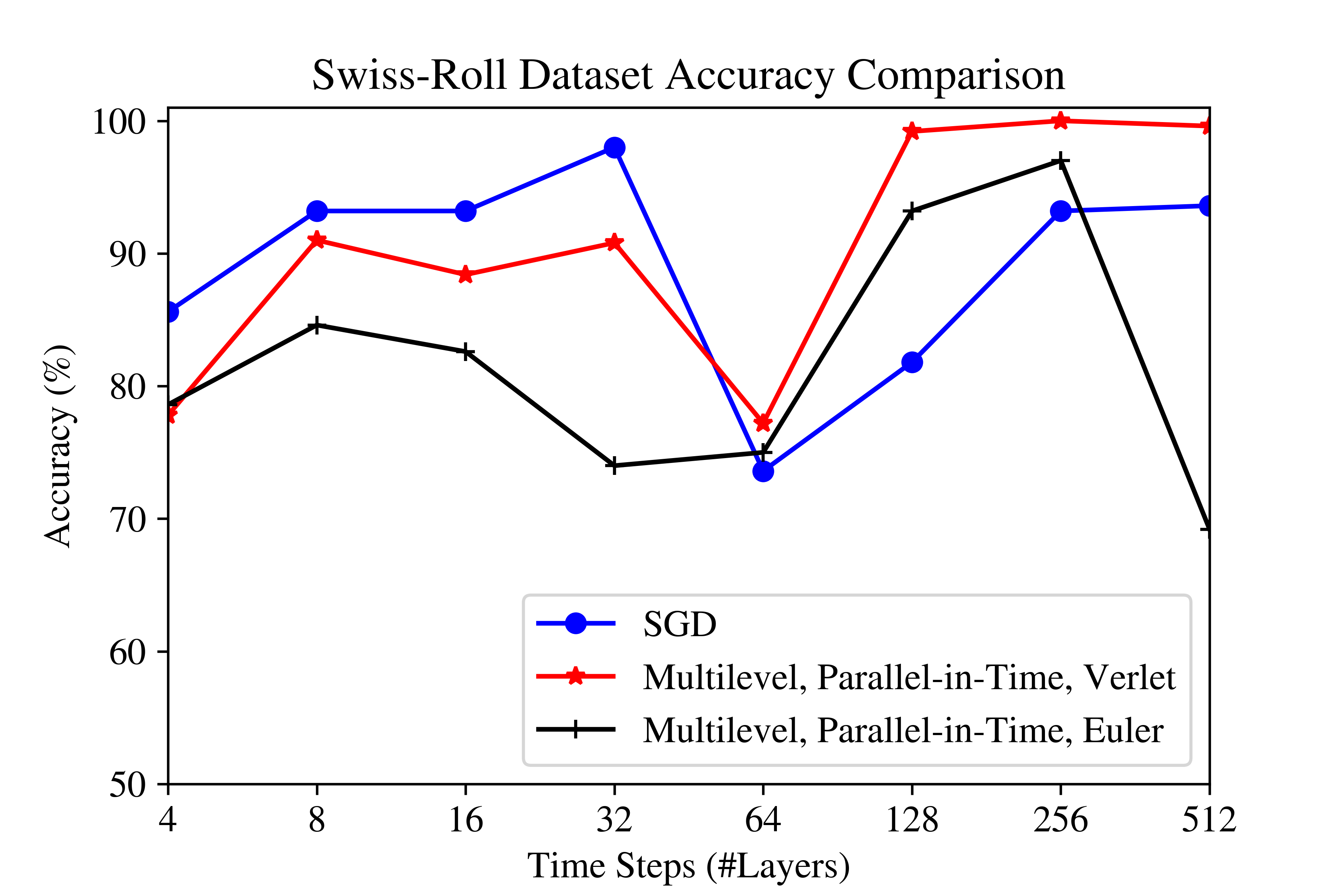}}
 \caption{Accuracy for the Swiss-Roll dataset when compared with a  data-parallel implementation of Stochastic Gradient Descent (SGD).}
 \label{fig:swiss accuracy}
 \end{center}
 \end{figure}

\begin{figure}[H]
 \begin{center}
 \centerline{\includegraphics[scale=0.75]{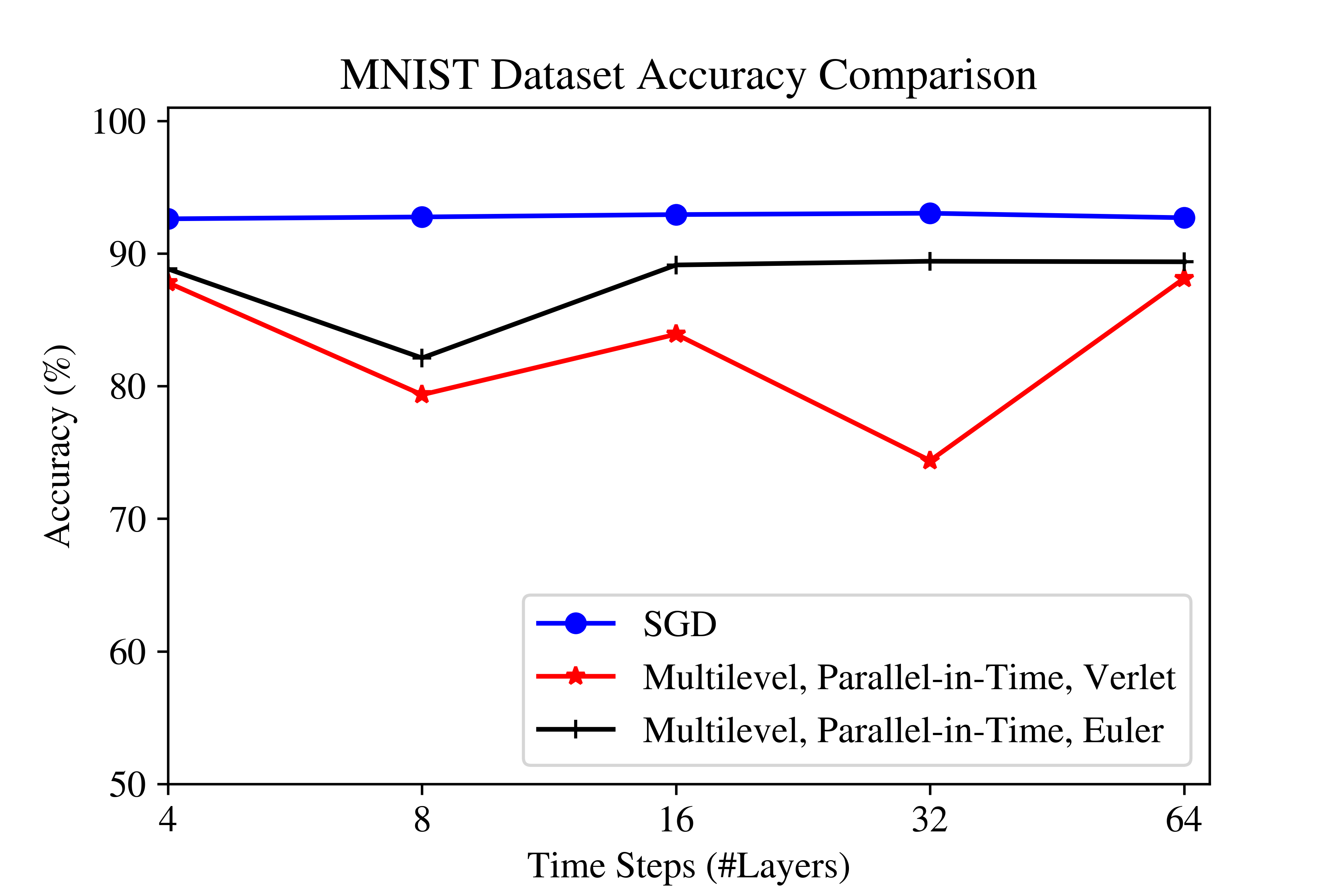}}
 \caption{Accuracy for the MNIST dataset when compared with a  data-parallel implementation of Stochastic Gradient Descent (SGD).}
 \label{fig:mnist accuracy}
 \end{center}
 \end{figure}

\end{document}